\documentclass[10pt,a4paper]{amsart}
\usepackage{amsmath,amsfonts,amsthm,amssymb,graphicx,hyperref}
\newcommand\F{\mathbb{F}}
\newcommand\FF{\mathcal{F}}
\newcommand\Haar{\mathrm{Haar}}
\newcommand\HH{\mathcal{H}}
\newcommand\PP{\mathcal{P}}
\newcommand\R{\mathbb{R}}
\newcommand\USp{\mathrm{USp}}
\newcommand\U{\mathrm{U}}
\newcommand\ZZ{\mathcal{Z}}
\def\e{\varepsilon}
\DeclareMathOperator{\cosec}{cosec}
\DeclareMathOperator{\Prob}{Prob}
\newtheorem{theorem}{Theorem}[section]
\newtheorem{corollary}[theorem]{Corollary}
\newtheorem{lemma}[theorem]{Lemma}
\newtheorem{proposition}[theorem]{Proposition}
\newtheorem*{Mertensconj}{The Mertens Conjecture}
\newtheorem*{Mertensconjfunctfield}{The Mertens Conjecture for Function Fields}
\newtheorem*{betaMertensconjfunctfield}{The $\beta$-Mertens Conjecture for Function Fields}
\theoremstyle{remark}
\newtheorem{remark}[theorem]{Remark}
\theoremstyle{definition}
\newtheorem{definition}[theorem]{Definition}
\begin{document}

\title{On the Mertens Conjecture for Function Fields}

\author{Peter Humphries}

\email{peterch@math.princeton.edu}

\address{Department of Mathematics, Princeton University, Princeton, New Jersey 08544, USA}



\thanks{This research was partially supported by an Australian Postgraduate Award.}

\begin{abstract}
We study the natural analogue of the Mertens conjecture in the setting of global function fields. Building on the work of Cha, we show that most hyperelliptic curves do not satisfy the Mertens conjecture, but that if we modify the Mertens conjecture to have a larger constant, then this modified conjecture is satisfied by a positive proportion of hyperelliptic curves.
\end{abstract}

\maketitle

\section{The Mertens Conjecture}

Let $\mu(n)$ denote the M\"{o}bius function, so that for a positive integer $n$,
	\[\mu(n) = \begin{cases}
1 & \text{if $n = 1$,}	\\
(-1)^t & \text{if $n$ is the product of $t$ distinct primes,}	\\
0 & \text{if $n$ is divisible by a perfect square,}
\end{cases}\]
and let
	\[M(x) = \sum_{n \leq x}{\mu(n)}
\]
be the summatory function of the M\"{o}bius function. In 1897, Mertens \cite{Mertens} calculated $M(x)$ from $x = 1$ up to $x = 10\,000$ and conjectured the following inequality.

\begin{Mertensconj}
For all $x \geq 1$, the summatory function of the M\"{o}bius function satisfies the inequality
\begin{equation}\label{Mertensconjecture}
\frac{|M(x)|}{\sqrt{x}} \leq 1.
\end{equation}
\end{Mertensconj}

The Mertens conjecture has several important consequences, the most notable of which are that the Riemann hypothesis is true and that all the zeroes of the Riemann zeta function, $\zeta(s)$, are simple; see \cite[\S 2]{Odlyzko} for further details.

In this article, we study the natural analogue of this conjecture in the setting of global function fields, that is, for nonsingular projective curves over finite fields. Let $C$ be a nonsingular projective curve of genus $g$ over a finite field $\F_q$ of characteristic $p$; we will assume throughout that $p$ is odd. For each effective divisor $N$ of $C$, we define the M\"{o}bius function of $C/\F_q$ to be
	\[\mu_{C/\F_q}(N) = \begin{cases}
1 & \text{if $N$ is the zero divisor,}	\\
(-1)^t & \text{if $N$ is the sum of $t$ distinct prime divisors of $C$,}	\\
0 & \text{if a prime divisor of $C$ divides $N$ with order at least $2$,}
\end{cases}\]
so that the summatory function of the M\"{o}bius function of $C/\F_q$ is
	\[M_{C/\F_q}(X) = \sum_{0 \leq \deg(N) \leq X - 1}{\mu_{C/\F_q}(N)},
\]
where $X$ is a positive integer. Our goal is to determine the validity of the following conjecture, the analogue of the Mertens conjecture in the function field setting.

\begin{Mertensconjfunctfield}
Let $C$ be a nonsingular projective curve over $\F_q$. The summatory function of the M\"{o}bius function of $C/\F_q$ satisfies
	\[\limsup_{X \to \infty} \frac{\left|M_{C/\F_q}(X)\right|}{q^{X / 2}} \leq 1.
\]
\end{Mertensconjfunctfield}

While the classical Mertens conjecture states that the inequality \eqref{Mertensconjecture} holds for all $x \geq 1$, the value $1$ on the right-hand side of \eqref{Mertensconjecture} is, in some sense, not particularly special. Indeed, Stieltjes \cite{Stieltjes} claimed in 1885 to have a proof that
\begin{equation}\label{Stieltjesbound}
M(x) = O\left(\sqrt{x}\right)
\end{equation}
without specifying an explicit constant, before later rescinding his claim, though he did postulate that \eqref{Mertensconjecture} was true. Similarly, von Sterneck \cite{vonSterneck} conjectured in 1912 that the stronger inequality
\begin{equation}\label{vonSterneckconj}
\frac{\left|M(x)\right|}{\sqrt{x}} \leq \frac{1}{2}
\end{equation}
holds for all $x \geq 200$, based on calculations of $M(x)$ up to $5 \, 000 \, 000$. So we may also consider the following variant of the Mertens conjecture for function fields.

\begin{betaMertensconjfunctfield}
Let $C$ be a nonsingular projective curve over $\F_q$, and let $\beta > 0$. The summatory function of the M\"{o}bius function of $C/\F_q$ satisfies
	\[\limsup_{X \to \infty} \frac{\left|M_{C/\F_q}(X)\right|}{q^{X / 2}} \leq \beta.
\]
\end{betaMertensconjfunctfield}

In spite of the numerical calculations of Mertens and  von Sterneck, the inequalities \eqref{Mertensconjecture} and \eqref{vonSterneckconj} are both now known to fail infinitely often. Odlyzko and te Riele \cite{Odlyzko} disproved the Mertens conjecture in 1985, and showed that
\begin{align*}
\limsup_{x \to \infty} \frac{M(x)}{\sqrt{x}} & > 1.06,	\\
\liminf_{x \to \infty} \frac{M(x)}{\sqrt{x}} & < -1.009.
\end{align*}
These bounds have since been improved to $1.218$ and $-1.229$ respectively by Kotnik and te Riele \cite{Kotnik}, and most recently to $1.6383$ and $-1.6383$ respectively by Best and Trudgian \cite{Best}. Stieltjes's claimed bound \eqref{Stieltjesbound} has yet to be disproved, although it seems likely that
\begin{align*}
\limsup_{x \to \infty} \frac{M(x)}{\sqrt{x}} & = \infty,	\\
\liminf_{x \to \infty} \frac{M(x)}{\sqrt{x}} & = -\infty.
\end{align*}
Indeed, Ingham \cite{Ingham} showed much earlier in 1942 that this follows from the assumption of the Riemann hypothesis and the linear independence over the rational numbers of the imaginary parts of the zeroes of $\zeta(s)$ in the upper half-plane. The latter hypothesis is known as the Linear Independence hypothesis; while there is as yet a lack of strong theoretical evidence for the falsity of the existence of any rational linear dependence between these imaginary parts, some limited numerical calculations have failed to find any such linear relations \cite{Bateman}, \cite{Best}. Most recently, Ingham's result has been refined conditionally by Ng \cite{Ng}, who has shown that the logarithmic density
	\[\delta\left(\PP_{\beta}\right) = \lim_{X \to \infty} \frac{1}{\log X} \int_{\PP_{\beta} \cap [1,X]}{\, \frac{dx}{x}}
\]
of the set
	\[\PP_{\beta} = \left\{x \in [1,\infty) : |M(x)| \leq \beta \sqrt{x}\right\}
\]
exists and satisfies the bound
	\[\delta\left(\PP_{\beta}\right) < 1
\]
for all $\beta > 0$, and also that for all sufficiently large $\beta$,
	\[\delta\left(\PP_{\beta}\right) > 0,
\]
all under the assumption of the Riemann hypothesis, the Linear Independence hypothesis, and that
	\[J_{-1}(T) = \sum_{0 < \gamma \leq T}{\left|\zeta'\left(1/2 + i \gamma\right)\right|^{-2}} \ll T,
\]
where the sum is over the nontrivial zeroes of $\zeta(s)$. Numerical calculations of Amir Akbary and Nathan Ng (personal communication) suggest that
	\[\delta\left(\PP_1\right) > 0.99999993366,
\]
so that the set of counterexamples of the Mertens conjecture, despite conditionally having strictly positive logarithmic density, is nevertheless extremely sparsely distributed in $[1,\infty)$.

In the function field setting, on the other hand, the situation is markedly different: if the zeroes of the zeta function $Z_{C/\F_q}(u)$ of $C/\F_q$ are not too poorly behaved, in the sense that $Z_{C/\F_q}(u)$ has only simple zeroes, then Cha \cite[Corollary 2.3]{Cha} has shown that
	\[\limsup_{X \to \infty} \frac{\left|M_{C/\F_q}(X)\right|}{q^{X / 2}}
\]
is bounded. Here $Z_{C/\F_q}(u)$ is defined initially for a complex variable $u$ in the open disc $|u| < q^{-1}$ via the absolutely convergent series
 	\[Z_{C/\F_q}(u) = \exp\left(\sum^{\infty}_{n = 1}{\# C\left(\F_{q^n}\right) \frac{u^n}{n}}\right).
\]
where $\# C\left(\F_{q^n}\right)$ denotes the number of points of $C$ in the field extension $\F_{q^n}$ of $\F_q$. This function extends meromorphically to the whole complex plane; indeed,
\begin{equation}\label{zetarational}
Z_{C/\F_q}(u) = \frac{P_{C/\F_q}(u)}{(1 - u)(1 - q u)},
\end{equation}
where $P_{C/\F_q}(u)$ is a polynomial of degree $2g$ with integer coefficients that factorises as
	\[P_{C/\F_q}(u) = \prod^{g}_{j = 1}{\left(1 - \gamma_j u\right) \left(1 - \overline{\gamma_j} u\right)}
\]
for some $\gamma_j = \sqrt{q} e^{i \theta\left(\gamma_j\right)}$ with $\theta\left(\gamma_j\right) \in [0,\pi]$; here $g$ is the genus of the curve $C$. The fact that each $\gamma_j$ satisfies $\left|\gamma_j\right| = \sqrt{q}$ is known as the Riemann hypothesis for function fields, and was proved by Weil in 1940 \cite{Weil}; as it is already known that the Riemann hypothesis for the Riemann zeta function is intimately connected to the growth of $M(x)$, we can immediately see the benefit of the function field setting. Furthermore, there are only finitely many zeroes of $Z_{C/\F_q}(u)$, so it is actually possible to confirm, given the zeta function of a curve $C/\F_q$, the following function field analogue of the Linear Independence hypothesis.

\begin{definition}
We say that $C$ satisfies the \emph{Linear Independence hypothesis}, which we abbreviate to LI, if the collection
	\[\pi, \theta\left(\gamma_1\right), \ldots, \theta\left(\gamma_g\right)
\]
is linearly independent over the rational numbers.
\end{definition}

See \cite[\S 6]{Kowalski} for examples of computationally determining whether a particular curve satisfies LI. Note also that the zeta function of a curve satisfying LI must only have simple zeroes.

In \cite{Cha}, Cha proves the following result about the maximal order of growth of $M_{C/\F_q}(X)$.

\begin{theorem}[Cha {\cite[Theorem 2.5]{Cha}}]\label{muLIboundsthm}
Suppose that $C$ is a nonsingular projective curve of genus $g \geq 1$ that satisfies \textup{LI}. Then
\begin{equation}\label{BLI}
\limsup_{X \to \infty} \frac{\left|M_{C/\F_q}(X)\right|}{q^{X / 2}} = 2 \sum^{g}_{j = 1}{\left|\frac{1}{{Z_{C/\F_q}}'\left(\gamma_j^{-1}\right)} \frac{\gamma_j}{\gamma_j - 1}\right|} < \infty.
\end{equation}
\end{theorem}

Cha also shows \cite[page 5]{Cha} that the case $g = 0$ is trivial, where $C = \mathbb{P}^1$ is the projective line, so that $C/\F_q$ is simply the rational function field $\F_q(t)$, for then
	\[M_{\F_q(t)}(X) = \begin{cases}
1 & \text{if $X = 1$,}	\\
-q & \text{if $X = 2$,}	\\
0 & \text{if $X \geq 3$.}
\end{cases}\]
Note that in this case
	\[Z_{\F_q(t)}(u) = \frac{1}{(1 - u)(1 - qu)}
\]
is the completed zeta function of the zeta function
	\[Z_{\F_q[t]}(u) = \frac{1}{1 - qu}
\]
of the ring $\F_q[t]$. One can define the M\"{o}bius function of a polynomial in $\F_q[t]$ in the same fashion as for a global function field $C/\F_q$ and show similarly that
	\[M_{\F_q[t]}(X) = \begin{cases}
1 & \text{if $X = 1$,}	\\
-q + 1 & \text{if $X \geq 2$;}
\end{cases}\]
see \cite[Chapter 2]{Rosen}.

So if a curve $C$ of genus $g \geq 1$ satisfies LI, then we need only determine the right-hand side of \eqref{BLI} in order to see whether $C/\F_q$ satisfies the Mertens conjecture. Unlike in the classical case, however, where we expect the Riemann zeta function to satisfy LI, there do exist curves $C$ that do not satisfy LI; furthermore, in the particular case when $Z_{C/\F_q}(u)$ has zeroes of multiple order, then the work of Cha \cite[Proposition 2.2]{Cha} indicates that
	\[\limsup_{X \to \infty} \frac{\left|M_{C/\F_q}(X)\right|}{q^{X / 2}} = \infty.
\]
Nevertheless, we can ensure that such curves are extremely rare by restricting to certain families of curves, namely hyperelliptic curves. With this family, we also have the added bonus of a framework for certain equidistribution results and connections to random matrix theory via the work of Katz and Sarnak \cite{Katz}.

We define this family of curves as follows: for $\F_{q^n}$ a finite field of odd characteristic, and for $g \geq 1$, let $f$ be a monic polynomial of degree $2g + 1$ with coefficients in $\F_{q^n}$ whose discriminant is nonzero; equivalently, let $f$ be a squarefree monic polynomial in $\F_{q^n}[x]$ of degree $2g + 1$. Each such polynomial $f$ thereby defines a hyperelliptic curve $C_f$ of genus $g$ over $\F_{q^n}$ via the affine model $y^2 = f(x)$. So we let $\HH_{2g + 1, q^n}$ denote the set of these hyperelliptic curves $C = C_f$ over $\F_{q^n}$. We are interested in properties of such curves $C$ shared by ``most'' $C \in \HH_{2g + 1, q^n}$. To define this notion, we consider $\HH_{2g + 1, q^n}$ as a probability space with the uniform probability measure, so that for a property $D$ of a hyperelliptic curve $C \in \HH_{2g + 1, q^n}$,
	\[\Prob_{\HH_{2g + 1, q^n}}\left(\textup{$C$ satisfies $D$}\right) = \frac{\#\left\{C \in \HH_{2g + 1, q^n} : \textup{$C$ satisfies $D$}\right\}}{\#\HH_{2g + 1, q^n}}.
\]

\begin{definition}\label{mostdef}
We say that \emph{most} hyperelliptic curves $C \in \HH_{2g + 1, q^n}$ have the property $D = \{D_n\}^{\infty}_{n = 1}$ as $n$ tends to infinity if
	\[\lim_{n \to \infty} \Prob_{\HH_{2g + 1, q^n}}\left(\textup{$C$ satisfies $D_n$}\right) = 1.
\]
\end{definition}

We are now able to state the main result of this article.

\begin{theorem}\label{mainthm}
Let $q$ and $g \geq 1$ be fixed. Then we have that
	\[\lim_{n \to \infty} \Prob_{\HH_{2g + 1, q^n}}\left(\textup{$C$ satisfies the $\beta$-Mertens conjecture}\right) = 0
\]
for $0 < \beta \leq 1$, whereas when $\beta > 1$,
	\[0 < \lim_{n \to \infty} \Prob_{\HH_{2g + 1, q^n}}\left(\textup{$C$ satisfies the $\beta$-Mertens conjecture}\right) < 1.
\]
That is, as $n$ tends to infinity, most hyperelliptic curves $C \in \HH_{2g + 1, q^n}$ do not satisfy the Mertens conjecture, but for any $\beta > 1$, a positive proportion of hyperelliptic curves satisfy the $\beta$-Mertens conjecture.
\end{theorem}

So although the Mertens conjecture is false for most hyperelliptic curves, the value $\beta = 1$ is critical, in that it is the greatest value of $\beta$ for which the $\beta$-Mertens conjecture is not satisfied for most hyperelliptic curves.

Results of this form were found by the author \cite{Humphries} in the low-genus case $g = 1$, so that $C$ is an elliptic curve: the author used a classification due to Waterhouse \cite{Waterhouse} of isogeny classes of elliptic curves over finite fields in terms of their Frobenius angles in order to determine explicitly the isogeny classes for which the Mertens conjecture holds, in the form of the following result.

\begin{theorem}[Humphries {\cite[Theorem 2.1]{Humphries}}]\label{mainellcurvethm}
Let $E$ be an elliptic curve over a finite field $\F_q$ of characteristic $p$. Then the Mertens conjecture for $E/\F_q$ is true if and only if the order of the finite field $q$ and the trace $a_E$ of the Frobenius endomorphism acting on $E$ over $\F_q$ satisfy precisely one of the following conditions:
\begin{enumerate}
\item[(1)]	$q = p^m$ with $a_E = 2$, where either $m$ is arbitrary and $p \neq 2$, or $m = 1$ and $p = 2$,
\item[(2)]	$q = p^m$ with $a_E = \sqrt{q}$, where $m$ is even and $p \not\equiv 1 \pmod{3}$,
\item[(3)]	$q = p^m$ with $a_E = 0$, where either $m$ is even and $p \not\equiv 1 \pmod{4}$, or $m$ is odd.
\end{enumerate}
In all these cases, we have that
	\[\limsup_{X \to \infty} \frac{\left|M_{E/\F_q}(X)\right|}{q^{X / 2}} = 1.
\]
\end{theorem}

While recent work of Howe, Nart, and Ritzenthaler \cite{Howe} classifies the isogeny classes of curves of genus two, there is as yet no such classification for curves of genus $g \geq 3$, so this method does not generalise to curves of large genus; indeed, a classification of isogeny classes of curves of a given genus $g$ would involve explicitly solving the (open) Schottky problem, namely giving an explicit description of all principally polarised abelian varieties that are Jacobian varieties of curves.

Instead, the methods for proving Theorem \ref{mainthm} involve relating the average
	\[\lim_{n \to \infty} \Prob_{\HH_{2g + 1, q^n}}\left(\textup{$C$ satisfies the $\beta$-Mertens conjecture}\right)
\]
to the Haar measure on a certain compact group of random matrices, then analysing the behaviour of the resulting probability value. These methods are closely related to the work of Cha \cite{Cha}, from which many of the results in this paper originate; in \cite{Cha}, Cha introduces the summatory function $M_{C/\F_q}(X)$ and studies a truncated form of the average of $\left|M_{C/\F_q}(X)\right| / q^{X / 2}$ over function fields,
	\[\lim_{n \to \infty} \frac{1}{\#\HH_{2g + 1, q^n}} \sum_{C \in \HH_{2g + 1, q^n}}{\left(\limsup_{X \to \infty} \frac{\left|M_{C/\F_q}(X)\right|}{q^{X / 2}}\right)}.
\]
Using random matrix methods, Cha is led to conjecture that this truncated average, after taking the limit as $n$ tends to infinity, is asymptotic to $c g^{1/4}$ in the limit as $g$ tends to infinity, where $c$ is a specific given constant.

\section{Preliminary Results}

We begin by converting this problem to a related problem for a certain family of random matrices. Our first step is to express the quantity
	\[B\left(C/\F_q\right) = \limsup_{X \to \infty} \frac{\left|M_{C/\F_q}(X)\right|}{q^{X / 2}}
\]
in the language of unitary symplectic matrices. Recall that the space of unitary symplectic matrices $\USp(2g)$ consists of $2g \times 2g$ matrices $U$ with complex entries satisfying $U^{\dagger} U = I$ and $U^T J U = J$, where
	\[J = \begin{pmatrix} 0 & I_g \\ - I_g & 0 \end{pmatrix}
\]
and $I_g$ denotes the $g \times g$ identity matrix. The eigenvalues of $U$ lie on the unit circle and come in complex conjugate pairs, so that we may order the eigenvalues $e^{i \theta_1}, \ldots, e^{i \theta_{2g}}$ such that $\theta_{j + g} = - \theta_j$ with $0 \leq \theta_j \leq \pi$ for $1 \leq j \leq g$. Conversely, given $\left(\theta_1, \ldots, \theta_g\right) \in [0,\pi]^g$, the diagonal matrix with diagonal entries $e^{i \theta_1}, \ldots, e^{i \theta_g}, e^{- i \theta_1}, \ldots, e^{- i \theta_g}$ lies in $\USp(2g)$. Thus the set of conjugacy classes $\USp(2g)^{\#}$ of $\USp(2g)$ corresponds to $[0,\pi]^g$.

\begin{definition}
For each $U \in \USp(2g)$, we define the \emph{characteristic polynomial} $\ZZ_U(\theta)$ for real $\theta$ by
	\[\ZZ_U(\theta) = \det\left(I - U e^{-i \theta}\right).
\]
Equivalently,
\begin{equation}\label{charpol}
\ZZ_U(\theta) = \prod^{2g}_{j = 1}{\left(1 - e^{i(\theta_j - \theta)}\right)} = 2^g \prod^{g}_{j = 1}{e^{i \theta}\left(\cos \theta - \cos \theta_j\right)}.
\end{equation}
\end{definition}

For a nonsingular projective curve $C$ over $\F_q$ of genus $g \geq 1$, there exists a conjugacy class $\vartheta\left(C/\F_q\right)$ in $\USp(2g)^{\#}$, called the unitarised Frobenius conjugacy class attached to $C/\F_q$, satisfying
\begin{equation}\label{charFrob}
\ZZ_{\vartheta\left(C/\F_q\right)}(\theta) = P_{C/\F_q}\left(\frac{e^{-i \theta}}{\sqrt{q}}\right) = \prod^{g}_{j = 1}{\left(1 - e^{i(\theta\left(\gamma_j\right) - \theta)}\right) \left(1 - e^{- i(\theta\left(\gamma_j\right) + \theta)}\right)}.
\end{equation}
That is, the eigenangles $\left(\theta_1, \ldots, \theta_g\right)$ corresponding to the unitarised Frobenius conjugacy class $\vartheta\left(C/\F_q\right)$ are precisely $\left(\theta\left(\gamma_1\right), \ldots, \theta\left(\gamma_g\right)\right)$, the angles of the inverse zeroes $\gamma_j = \sqrt{q} e^{i \theta(\gamma_j)}$, $1 \leq j \leq g$, of $Z_{C/\F_q}(u)$.

We require an expression for $B\left(C/\F_q\right)$ in terms of $\ZZ_{\vartheta\left(C/\F_q\right)}(\theta)$ in the large $q$ limit. For $U \in \USp(2g)$, we define the function $\varphi(U)$ by
	\[\varphi(U) = 2 \sum^{g}_{j = 1}{\frac{1}{\left|{\ZZ_U}'\left(\theta_j\right)\right|}},
\]
where $e^{i \theta_1}, \ldots, e^{i \theta_g}, e^{- i \theta_1}, \ldots, e^{i \theta_g}$ are the eigenvalues of $U$, with $0 \leq \theta_j \leq \pi$ for $1 \leq j \leq g$. We observe that $\varphi$ depends only on the conjugacy class $\left(\theta_1, \ldots, \theta_g\right)$ of $U$, and that $\varphi$ is always nonnegative, though it blows up if $U$ has a repeated eigenvalue. Note, however, that the set of matrices in $\USp(2g)$ with repeated eigenvalues has measure zero with respect to the normalised Haar measure on $\USp(2g)$.

\begin{lemma}[Cha {\cite[Equation (26)]{Cha}}]\label{muBqasymp}
Suppose that $C$ satisfies \textup{LI}. Then we have that
	\[\left(1 - \frac{1}{\sqrt{q}}\right) \varphi\left(\vartheta\left(C/\F_q\right)\right) \leq B\left(C/\F_q\right) \leq \left(1 + \frac{1}{\sqrt{q}}\right) \varphi\left(\vartheta\left(C/\F_q\right)\right).
\]
\end{lemma}

Of course, we cannot say which individual curves satisfy LI without studying their zeroes. Nevertheless, we have the following remarkable result for the family of hyperelliptic curves.

\begin{theorem}[Cha--Chavdarov--Kowalski {\cite[Theorem 3.1]{Cha}, \cite{Chavdarov}, \cite{Kowalski}}]\label{Kowalskithm}
Let $q$ and $g \geq 1$ be fixed. Then
	\[\lim_{n \to \infty} \Prob_{\HH_{2g + 1, q^n}}\left(\textup{$C$ satisfies LI}\right) = 1.
\]
That is, as $n$ tends to infinity, most hyperelliptic curves $C \in \HH_{2g + 1, q^n}$ satisfy \textup{LI}.
\end{theorem}

This result is our reason for restricting ourselves to the family of hyperelliptic curves, as well as the usefulness of the following equidistribution theorem.

\begin{theorem}[Deligne's Equidistribution Theorem {\cite[Theorem 10.8.2]{Katz}}]
Let $f$ be a continuous function on $\USp(2g)$ that is central, so that $f$ is dependent only on the conjugacy class $\left(\theta_1, \ldots, \theta_g\right)$ of each matrix $U \in \USp(2g)$. Let $q$ and $g \geq 1$ be fixed. Then
	\[\lim_{n \to \infty} \frac{1}{\# \HH_{2g + 1, q^n}} \sum_{C \in \HH_{2g + 1, q^n}}{f\left(\vartheta\left(C/\F_{q^n}\right)\right)} = \int_{\USp(2g)}{f(U) \, d\mu_{\Haar}(U)},
\]
where $\mu_{\Haar}$ is the normalised Haar measure on $\USp(2g)$.
\end{theorem}

Equivalently, consider the sequence of probability measures
	\[\Prob_{\HH_{2g + 1, q^n}} = \frac{1}{\# \HH_{2g + 1, q^n}} \sum_{C \in \HH_{2g + 1, q^n}}{\delta_{\vartheta\left(C/\F_{q^n}\right)}}
\]
on $\USp(2g)$, where $\delta_{U^{\#}}$ is a point mass at a conjugacy class $U^{\#} \in \USp(2g)^{\#}$. Then Deligne's equidistribution theorem merely states that the sequence of probability measures $\Prob_{\HH_{2g + 1, q^n}}$ converges weakly to the probability measure
	\[\Prob_{\USp(2g)} = \mu_{\Haar}
\]
as $n$ tends to infinity. By applying the Portmanteau theorem to the sequence of probability measures $\Prob_{\HH_{2g + 1, q^n}}$, we obtain an equivalent reformulation of Deligne's equidistribution theorem.

\begin{corollary}\label{Delignecor}
For fixed $g \geq 1$,
	\[\lim_{n \to \infty} \Prob_{\HH_{2g + 1, q^n}}\left(\vartheta\left(C/\F_{q^n}\right) \in A\right) = \Prob_{\USp(2g)}\left(U \in A\right)
\]
for any Borel set $A \subset \USp(2g)$ whose boundary has Haar measure zero.
\end{corollary}

\begin{remark}
In fact, Deligne's equidistribution theorem holds not only for fixed $q$ in the limit as $n$ tends to infinity, but rather for \emph{any} sequence of prime powers $q$ tending to infinity, with $\HH_{2g + 1, q}$ in place of $\HH_{2g + 1, q^n}$. However, Theorem \ref{Kowalskithm} requires the restriction that $q$ be fixed, which is why we have this condition in Theorem \ref{mainthm}. It would be of interest to determine whether this restriction could be removed, so that Theorem \ref{mainthm} would hold for any sequence of prime powers $q$ tending to infinity.
\end{remark}

One can calculate the Haar measure for $\USp(2g)$ precisely by using the following formula to convert it into an integral over $[0,\pi]^g$.

\begin{proposition}[Weyl Integration Formula {\cite[\S 5.0.4]{Katz}}]
Let $f$ be a bounded, Borel-measurable complex-valued central function on $\USp(2g)$. Then
	\[\int_{\USp(2g)}{f(U) \, d\mu_{\Haar}(U)} = \int^{\pi}_{0} \hspace{-0.2cm} \cdots \hspace{-0.1cm} \int^{\pi}_{0} f\left(\theta_1, \ldots, \theta_g\right) \, d\mu_{\USp}\left(\theta_1, \ldots, \theta_g\right),
\]
where
\begin{equation}\label{muUSp}
d\mu_{\USp}\left(\theta_1, \ldots, \theta_g\right) = \frac{2^{g^2}}{g! \pi^g} \prod_{1 \leq j < k \leq g}{\left(\cos \theta_k - \cos \theta_j\right)^2} \prod^{g}_{\ell = 1}{\sin^2 \theta_{\ell}} \, d\theta_1 \cdots d\theta_g.
\end{equation}
\end{proposition}

To make use of Corollary \ref{Delignecor}, we need to ensure that the boundaries of the sets with which we work have Haar measure zero.

\begin{lemma}\label{muboundary}
Let $A$ be an interval in $\R$. Then the boundary of the set
	\[\left\{U \in \USp(2g) : \varphi(U) \in A\right\}
\]
has Haar measure zero.
\end{lemma}

\begin{proof}
By differentiating \eqref{charpol}, we have that
\begin{equation}\label{varphitheta}
\varphi(U) = \varphi\left(\theta_1, \ldots, \theta_g\right) = \frac{1}{2^{g - 1}} \sum^{g}_{j = 1}{\cosec \theta_j \prod^{g}_{\substack{k = 1 \\ k \neq j}}{\frac{1}{\left|\cos \theta_k - \cos \theta_j\right|}}}.
\end{equation}
So by the Weyl integration formula, we must show that for any interval $A$, the boundary of the set
	\[\left\{\left(\theta_1, \ldots, \theta_g\right) \in [0,\pi]^g : \varphi\left(\theta_1, \ldots, \theta_g\right) \in A\right\}
\]
has $\mu_{\USp}$-measure zero. Observe that $\mu_{\USp}$ is absolutely continuous with respect to the Lebesgue measure on $[0,\pi]^g$, and hence the sets
	\[\left\{\left(\theta_1, \ldots, \theta_g\right) \in [0,\pi]^g : \textup{$\theta_j = \theta_k$ for some $1 \leq j < k \leq g$}\right\}
\]
and
	\[\left\{\left(\theta_1, \ldots, \theta_g\right) \in [0,\pi]^g : \textup{$\theta_j \in \{0,\pi\}$ for some $1 \leq j \leq g$}\right\}
\]
have $\mu_{\USp}$-measure zero; furthermore, for each permutation $\sigma$ of $\{1, \ldots, g\}$, the function $\varphi$ is continuous on the set
	\[\left\{\left(\theta_1, \ldots, \theta_g\right) \in [0,\pi]^g : 0 < \theta_{\sigma(1)} < \ldots < \theta_{\sigma(g)} < \pi\right\}.
\]
It therefore suffices to show that for each $a \in \R$ and for each $\sigma \in S_g$, the set
	\[\left\{\left(\theta_1, \ldots, \theta_g\right) \in [0,\pi]^g : \varphi\left(\theta_1, \ldots, \theta_g\right) = a, \; 0 < \theta_{\sigma(1)} < \ldots < \theta_{\sigma(g)} < \pi\right\}
\]
has $\mu_{\USp}$-measure zero. But in the region where $0 < \theta_{\sigma(1)} < \ldots < \theta_{\sigma(g)} < \pi$, the expression \eqref{varphitheta} shows that the function $\varphi\left(\theta_1, \ldots, \theta_g\right)$ is not only continuous but real analytic and non-uniformly constant. As the zero set of a non-uniformly zero real analytic function has Lebesgue measure zero, and $\mu_{\USp}$ is absolutely continuous with respect to the Lebesgue measure, we obtain the result.
\end{proof}

\section{Proof of Theorem \ref{mainthm}}

We have now developed the necessary machinery needed in order to study the limit as $n$ tends to infinity of the average
	\[\Prob_{\HH_{2g + 1, q^n}}\left(\textup{$C$ satisfies the $\beta$-Mertens conjecture}\right).
\]
For brevity's sake, we write this average as
	\[\Prob_{\HH_{2g + 1, q^n}}\left(B\left(C/\F_{q^n}\right) \leq \beta\right).
\]
We also write $C \in \textup{LI}$ if $C$ satisfies LI, and conversely if $C$ does not satisfy LI, we write $C \notin \textup{LI}$.

\begin{proposition}\label{Mertensaverage}
For $\beta > 0$, we have that
\begin{equation}\label{Mertensaverageeq}
\lim_{n \to \infty} \Prob_{\HH_{2g + 1, q^n}}\left(B\left(C/\F_{q^n}\right) \leq \beta\right) = \Prob_{\USp(2g)}\left(\varphi(U) \leq \beta\right).
\end{equation}
\end{proposition}

Note that \eqref{Mertensaverageeq} is equivalent to
\begin{equation}\label{Mertensaverageeq2}
\lim_{n \to \infty} \Prob_{\HH_{2g + 1, q^n}}\left(B\left(C/\F_{q^n}\right) > \beta\right) = \Prob_{\USp(2g)}\left(\varphi(U) > \beta\right).
\end{equation}
for $\beta > 0$. If we let
	\[B^T\left(C/\F_q\right) = \min\left\{B\left(C/\F_q\right), T\right\}
\]
for $T > 0$, then by integrating \eqref{Mertensaverageeq2} with respect to $\beta$ from $0$ to $T$ and taking the limit as $T$ tends to infinity, we find via the dominated convergence theorem and Fubini's theorem that
	\[\lim_{T \to \infty} \lim_{n \to \infty} \frac{1}{\#\HH_{2g + 1, q^n}} \sum_{C \in \HH_{2g + 1, q^n}}{B^T\left(C/\F_{q^n}\right)} = \int_{\USp(2g)}{\varphi(U) \, d\mu_{\Haar}(U)},
\]
thereby obtaining a slightly modified version of a result of Cha \cite[Theorem 3.3]{Cha}.

\begin{proof}
For any $\e > 0$ with $\e < \beta$, let
\begin{align*}
A & = \left\{C \in \HH_{2g + 1, q^n} : B\left(C/\F_{q^n}\right) \leq \beta, \; \varphi\left(\vartheta\left(C/\F_{q^n}\right)\right) \leq \beta, \; C \in \textup{LI}\right\},	\\
A_1 & = \left\{C \in \HH_{2g + 1, q^n} : B\left(C/\F_{q^n}\right) \leq \beta\right\},	\\
A_2 & = \left\{C \in \HH_{2g + 1, q^n} : B\left(C/\F_{q^n}\right) \leq \beta, \; C \notin \textup{LI}\right\},	\\
A_3 & = \left\{C \in \HH_{2g + 1, q^n} : B\left(C/\F_{q^n}\right) \leq \beta, \; \beta < \varphi\left(\vartheta\left(C/\F_{q^n}\right)\right) \leq \beta + \e, \; C \in \textup{LI}\right\},	\\
A_4 & = \left\{C \in \HH_{2g + 1, q^n} : B\left(C/\F_{q^n}\right) \leq \beta, \; \varphi\left(\vartheta\left(C/\F_{q^n}\right)\right) > \beta + \e, \; C \in \textup{LI}\right\},	\\
A_{1'} & = \left\{C \in \HH_{2g + 1, q^n} : \varphi\left(\vartheta\left(C/\F_{q^n}\right)\right) \leq \beta\right\},	\\
A_{2'} & = \left\{C \in \HH_{2g + 1, q^n} : \varphi\left(\vartheta\left(C/\F_{q^n}\right)\right) \leq \beta, \; C \notin \textup{LI}\right\},	\\
A_{3'} & = \left\{C \in \HH_{2g + 1, q^n} : B\left(C/\F_{q^n}\right) > \beta, \; \beta - \e \leq \varphi\left(\vartheta\left(C/\F_{q^n}\right)\right) \leq \beta, \; C \in \textup{LI}\right\},	\\
A_{4'} & = \left\{C \in \HH_{2g + 1, q^n} : B\left(C/\F_{q^n}\right) > \beta, \; \varphi\left(\vartheta\left(C/\F_{q^n}\right)\right) < \beta - \e, \; C \in \textup{LI}\right\}.
\end{align*}
Then we have that
\begin{align*}
A_1 & = A \sqcup A_2 \sqcup A_3 \sqcup A_4,	\\
A_{1'} & = A \sqcup A_{2'} \sqcup A_{3'} \sqcup A_{4'},
\end{align*}
where the unions are all disjoint, and consequently
	\[\# A_1 = \# A_{1'} + \# A_2 - \# A_{2'} + \# A_3 - \# A_{3'} + \# A_4 - \# A_{4'}.
\]
By Deligne's equidistribution theorem and Lemma \ref{muboundary},
	\[\lim_{n \to \infty} \frac{\# A_{1'}}{\# \HH_{2g + 1, q^n}} = \Prob_{\USp(2g)}\left(\varphi(U) \leq \beta\right),
\]
while Theorem \ref{Kowalskithm} implies that
	\[\lim_{n \to \infty} \frac{\# A_2}{\# \HH_{2g + 1, q^n}} = \lim_{n \to \infty} \frac{\# A_{2'}}{\# \HH_{2g + 1, q^n}} = 0.
\]
Next, we note that
	\[A_3 \sqcup A_{3'} \subset \left\{C \in \HH_{2g + 1, q^n} : \beta - \e \leq \varphi\left(\vartheta\left(C/\F_{q^n}\right)\right) \leq \beta + \e\right\},
\]
and hence
	\[\limsup_{n \to \infty} \frac{\left|\# A_3 - \# A_{3'}\right|}{\# \HH_{2g + 1, q^n}} \leq \Prob_{\USp(2g)}\left(\beta - \e \leq \varphi(U) \leq \beta + \e\right).
\]
by Deligne's equidistribution theorem and Lemma \ref{muboundary}. Finally, Lemma \ref{muBqasymp} implies that
	\[A_4 \subset \left\{C \in \HH_{2g + 1, q^n} : \beta + \e < \varphi\left(\vartheta\left(C/\F_{q^n}\right)\right) \leq \frac{1}{1 - q^{-n/2}} \beta, \; C \in \textup{LI}\right\},
\]
which is empty for all $n \geq 2 \log_q(\beta/\e + 1)$, and similarly that
	\[A_{4'} \subset \left\{C \in \HH_{2g + 1, q^n} : \frac{1}{1 + q^{-n/2}} \beta < \varphi\left(\vartheta\left(C/\F_{q^n}\right)\right) <  \beta - \e, \; C \in \textup{LI}\right\},
\]
which is empty for all $n \geq 2 \log_q(\beta/\e - 1)$. So
	\[\lim_{n \to \infty} \frac{\# A_4}{\# \HH_{2g + 1, q^n}} = \lim_{n \to \infty} \frac{\# A_{4'}}{\# \HH_{2g + 1, q^n}} = 0.
\]
So we have shown that for any $\e > 0$ with $\e < \beta$,
\begin{multline*}
\limsup_{n \to \infty} \left|\Prob_{\HH_{2g + 1, q^n}}\left(B\left(C/\F_{q^n}\right) \leq \beta\right) - \Prob_{\USp(2g)}\left(\varphi(U) \leq \beta\right)\right|	\\
\leq \Prob_{\USp(2g)}\left(\beta - \e \leq \varphi(U) \leq \beta + \e\right).
\end{multline*}
As $\e > 0$ was arbitrary, and
	\[\lim_{\e \to 0} \Prob_{\USp(2g)}\left(\beta - \e \leq \varphi(U) \leq \beta + \e\right) = \Prob_{\USp(2g)}\left(\varphi(U) = \beta\right) = 0
\]
by Lemma \ref{muboundary}, we obtain the result.
\end{proof}

So in order to prove Theorem \ref{mainthm}, we must show that for each fixed $g \geq 1$,
	\[\Prob_{\USp(2g)}\left(\varphi(U) \leq \beta\right) = 0
\]
for $0 < \beta \leq 1$, whereas for any $\beta > 1$,
	\[0 < \Prob_{\USp(2g)}\left(\varphi(U) \leq \beta\right) < 1.
\]
This follows from the following result, which we prove in Section \ref{unitarysection} in a more general form.

\begin{proposition}\label{unitarysympthm}
Let $U \in \USp(2g)$ be a unitary symplectic matrix with eigenvalues $e^{i \theta_1}, \ldots, e^{i \theta_g}, e^{- i \theta_1}, \ldots, e^{i \theta_g}$, with $0 \leq \theta_j \leq \pi$ for $1 \leq j \leq g$. Then the global minimum of
	\[\varphi(U) = 2 \sum^{g}_{j = 1}{\frac{1}{|{\ZZ_U}'(\theta_j)|}}
\]
occurs precisely at the set of points
	\[\left(\widetilde{\theta}_{\sigma(1)}, \ldots, \widetilde{\theta}_{\sigma(g)}\right),
\]
where $\sigma$ is a permutation on $\{1, \ldots, g\}$, and
	\[\left(\widetilde{\theta}_1, \ldots, \widetilde{\theta}_g\right) = \left(\frac{\pi}{2g}, \frac{3 \pi}{2g}, \ldots, \frac{(2g - 1) \pi}{2g}\right).
\]
Furthermore,
	\[\varphi\left(\widetilde{\theta}_{\sigma(1)}, \ldots, \widetilde{\theta}_{\sigma(g)}\right) = 1.
\]
\end{proposition}

\begin{proof}[Proof of Theorem \ref{mainthm}]
As Proposition \ref{unitarysympthm} implies that the set
	\[\left\{\left(\theta_1, \ldots, \theta_g\right) \in [0,\pi]^g : \varphi\left(\theta_1, \ldots, \theta_g\right) \leq 1\right\}
\]
is finite, we must have that
	\[\Prob_{\USp(2g)}\left(\varphi(U) \leq \beta\right) = 0
\]
for all $0 < \beta \leq 1$ via the Weyl integration formula and the fact that the measure $\mu_{\USp}$ is atomless, with $\mu_{\USp}$ as in \eqref{muUSp}. To prove that
	\[0 < \Prob_{\USp(2g)}\left(\varphi(U) \leq \beta\right) < 1
\]
for $\beta > 1$, we note that from Proposition \ref{unitarysympthm}, the equality $\varphi\left(\theta_1, \ldots, \theta_g\right) = 1$ is attained at the point
	\[\left(\widetilde{\theta}_1, \ldots, \widetilde{\theta}_g\right) = \left(\frac{\pi}{2g}, \frac{3 \pi}{2g}, \ldots, \frac{(2g - 1) \pi}{2g}\right),
\]
which lies in the region $0 < \theta_1 < \ldots < \theta_g < \pi$. As $\varphi$ is real analytic and non-uniformly constant in this region, there must exist an open neighbourhood of $\left(\widetilde{\theta}_1, \ldots, \widetilde{\theta}_g\right)$ in this region where $1 \leq \varphi\left(\theta_1, \ldots, \theta_g\right) \leq \beta$. This open neighbourhood must have positive $\mu_{\USp}$-measure, as $d\mu_{\USp}\left(\theta_1, \ldots, \theta_g\right)$ does not vanish on open subsets of $[0,\pi]^g$. Consequently,
	\[\Prob_{\USp(2g)}\left(\varphi(U) \leq \beta\right) > 0.
\]
On the other hand, we must also have that
	\[\Prob_{\USp(2g)}\left(\varphi(U) \leq \beta\right) < 1,
\]
as $\varphi$ blows up when $\theta_j = \theta_k$ for any $j \neq k$, and so for any such point there exists some open neighbourhood with $\varphi\left(\theta_1, \ldots, \theta_g\right) > \beta$ within this neighbourhood.
\end{proof}

It is worth noting that for $\beta > 1$, the quantity
	\[\Prob_{\USp(2g)}\left(\varphi(U) \leq \beta\right)
\]
is in fact strictly increasing as a function of $\beta$; this follows from the argument above together with the fact that there exists a point $\left(\theta_1, \ldots, \theta_g\right)$ for which $\varphi\left(\theta_1, \ldots, \theta_g\right) = \beta$, which follows from the mean value theorem.

\section{Generalisations}

The definition of the summatory function of the M\"{o}bius function of $C/\F_q$ studied in this article involves a sum over all effective divisors $N$ of $C$ with $0 \leq \deg(N) \leq X - 1$, following the same definition as used previously in the literature in \cite{Cha} and \cite{Humphries} (though note that Cha in \cite{Cha} defines this sum to be over all $N$ with $0 \leq \deg(N) \leq X$, and then normalises $B(C/\F_q)$ by a factor of $1/\sqrt{q}$ to compensate for this alteration). On the other hand, it is more common to define the summatory functions of arithmetic functions of function fields as only involving a sum over effective divisors $N$ of a fixed degree, $\deg(N) = X$; for example, see \cite[Chapter 17]{Rosen}. Here we shall observe that this distinction is moot: we will show that Theorem \ref{mainthm} remains valid after replacing the $\beta$-Mertens conjecture
	\[\limsup_{X \to \infty} \frac{\left|M_{C/\F_q}(X)\right|}{q^{X / 2}} = \limsup_{X \to \infty} \frac{1}{q^{X / 2}} \left|\sum_{0 \leq \deg(N) \leq X - 1}{\mu_{C/\F_q}(N)}\right| \leq \beta
\]
by the ``localised'' $\beta$-Mertens conjecture
	\[\limsup_{X \to \infty} \frac{1}{q^{X / 2}} \left|\sum_{\deg(N) = X - 1}{\mu_{C/\F_q}(N)}\right| \leq \beta.
\]
That is, we need not average over all effective divisors of degree at most $X - 1$, but merely study the growth of the M\"{o}bius function of the effective divisors of degree $X - 1$. To prove this, we only require the following analogue of Lemma \ref{muBqasymp}, for then the associated analogue of Proposition \ref{Mertensaverage} holds with only trivial modifications to the proof.

\begin{lemma}[cf.\ Lemma \ref{muBqasymp}]
Suppose that $C$ satisfies \textup{LI}. Then
\begin{align*}
\limsup_{X \to \infty} \frac{1}{q^{X / 2}} \left|\sum_{\deg(N) = X - 1}{\mu_{C/\F_q}(N)}\right| & \geq \left(1 - \frac{1}{\sqrt{q}}\right)^2 \varphi\left(\vartheta\left(C/\F_q\right)\right),	\\
\limsup_{X \to \infty} \frac{1}{q^{X / 2}} \left|\sum_{\deg(N) = X - 1}{\mu_{C/\F_q}(N)}\right| & \leq \left(1 + \frac{1}{\sqrt{q}}\right)^2 \varphi\left(\vartheta\left(C/\F_q\right)\right),
\end{align*}
\end{lemma}

\begin{proof}
From the method of proof of \cite[Proposition 2.2]{Cha}, if $Z_{C/\F_q}(u)$ has only simple zeroes, then as $X$ tends to infinity,
	\[\sum_{\deg(N) = X - 1}{\mu_{C/\F_q}(N)} = - 2 \Re\left(\sum^{g}_{j = 1}{\frac{\gamma_j^X}{{Z_{C/\F_q}}'\left(\gamma_j^{-1}\right)}}\right) + O(1),
\]
and so if $C$ satisfies LI, the Kronecker--Weyl theorem and the fact that $\gamma_j^X = q^{X / 2} e^{i X \theta\left(\gamma_j\right)}$ imply that
	\[\limsup_{X \to \infty} \frac{1}{q^{X / 2}} \left|\sum_{\deg(N) = X - 1}{\mu_{C/\F_q}(N)}\right| = 2 \sum^{g}_{j = 1}{\frac{1}{\left|{Z_{C/\F_q}}'\left(\gamma_j^{-1}\right)\right|}}.
\]
By \eqref{zetarational}, \eqref{charFrob}, and the fact that $\gamma_j = \sqrt{q} e^{i \theta\left(\gamma_j\right)}$, we have that
	\[\frac{1}{{Z_{C/\F_q}}'\left(\gamma_j^{-1}\right)} = i \frac{q + 1 - 2 \sqrt{q} \cos \theta\left(\gamma_j\right)}{q e^{2 i \theta\left(\gamma_j\right)}} \frac{1}{{\ZZ_{\vartheta\left(C/\F_q\right)}}'\left(\theta\left(\gamma_j\right)\right)}.
\]
As
	\[\left(\sqrt{q} - 1\right)^2 \leq \left|q + 1 - 2 \sqrt{q} \cos \theta\left(\gamma_j\right)\right| \leq \left(\sqrt{q} + 1\right)^2,
\]
we can take absolute values and then sum from $j = 1$ to $j = g$, yielding the result.
\end{proof}

We also note that we can prove a weaker form of Theorem \ref{mainthm} for families of curves other than hyperelliptic curves. Indeed, let $\FF = \{\FF_q\}$ be a family of curves indexed by a set of prime powers $q$ tending to infinity, such that each $\FF_q$ consists of a finite set of nonsingular projective curves over $\F_q$; note that we do not require that the sequence $q$ consist only of powers of a single fixed prime, and also that there is no restriction whatsoever on the genus of a curve in each $\FF_q$. For a property $D$ of a curve $C \in \FF_q$, we define the probability
	\[\Prob_{\FF_q}\left(\textup{$C$ satisfies $D$}\right) = \frac{\#\left\{C \in \FF_q : \textup{$C$ satisfies $D$}\right\}}{\#\FF_q}.
\]

\begin{theorem}\label{extrathm}
Suppose that
	\[\lim_{q \to \infty} \Prob_{\FF_q}\left(\textup{$C$ satisfies LI}\right) = 1.
\]
Then
\begin{equation}\label{FFq<1}
\lim_{q \to \infty} \Prob_{\FF_q}\left(B(C/\F_q) < 1\right) = 0. 
\end{equation}
Furthermore, there exists a family $\FF$ for which $C$ satisfies \textup{LI} for all $C \in \FF$, but with
\begin{equation}\label{FFq=1}
\lim_{q \to \infty} \Prob_{\FF_q}\left(B(C/\F_q) = 1\right) = 1, 
\end{equation}
and similarly for every fixed $b > a > 1$ there exists a family $\FF$ for which $C$ satisfies \textup{LI} for all $C \in \FF$, but with
\begin{equation}\label{FFqab}
\lim_{q \to \infty} \Prob_{\FF_q}\left(a < B(C/\F_q) < b\right) = 1.
\end{equation}
\end{theorem}

\begin{proof}
For \eqref{FFq<1}, we copy the proof of Proposition \ref{Mertensaverage} with $\FF_q$ replacing $\HH_{2g + 1, q^n}$, taking the limit as $q$ tends to infinity as opposed to $n$ tending to infinity, and choosing $\beta = 1 - \delta$ for fixed $0 < \delta < 1$, so that $0 < \e < 1 - \delta$. We treat the sets $A_2, A_{2'}, A_4, A_{4'}$ using the same method as in Proposition \ref{Mertensaverage}, obtaining
	\[\lim_{q \to \infty} \frac{\# A_2}{\# \FF_q} = \lim_{q \to \infty} \frac{\# A_{2'}}{\# \FF_q} = \lim_{q \to \infty} \frac{\# A_4}{\# \FF_q} = \lim_{q \to \infty} \frac{\# A_{4'}}{\# \FF_q} = 0.
\]
For $A_3$ and $A_3'$, we have that
	\[A_3 \sqcup A_{3'} \subset \left\{C \in \FF_q : 1 - \delta - \e \leq \varphi\left(\vartheta\left(C/\F_q\right)\right) \leq 1 - \delta + \e\right\},
\]
and as $1 - \delta + \e < 1$, Proposition \ref{unitarysympthm} implies that this set is empty, so that
	\[\lim_{q \to \infty} \frac{\# A_3}{\# \FF_q} = \lim_{q \to \infty} \frac{\# A_{3'}}{\# \FF_q} = 0.
\]
Finally, for $A_{1'}$ we write
	\[A_{1'} = A_{2'} \sqcup A_{5'}
\]
with
	\[A_{5'} = \left\{C \in \FF_q : \varphi\left(\vartheta\left(C/\F_q\right)\right) \leq 1 - \delta, \; C \in \textup{LI}\right\}.
\]
Again,
	\[\lim_{q \to \infty} \frac{\# A_{2'}}{\# \FF_q} = 0,
\]
and Proposition \ref{unitarysympthm} shows that $A_{5'}$ is empty, so that
	\[\lim_{q \to \infty} \frac{\# A_{5'}}{\# \FF_q} = 0
\]
as well. Thus
	\[\Prob_{\FF_q}\left(B(C/\F_q) \leq 1 - \delta\right) = \lim_{q \to \infty} \frac{\# A_1}{\# \FF_q} = 0,
\]
and as $\delta > 0$ was arbitrary, we obtain \eqref{FFq<1}.

For \eqref{FFq=1}, we take $q$ to be any odd prime power and $\FF_q$ to consist solely of the elliptic curve $E$ over $\F_q$ whose trace of the Frobenius is equal to $2$; by the proof of \cite[Theorem 4.1]{Waterhouse}, such an $E$ exists and satisfies LI, and by Theorem \ref{mainellcurvethm}, $B(E/\F_q) = 1$.

Finally, for \eqref{FFqab}, we take $q = p^m$ for some prime $p$ and we take $\FF_q$ to consist of the set of elliptic curves $E$ over $\F_q$ whose trace $a_E$ of the Frobenius is an integer satisfying $a_E \not\equiv 0 \pmod{p}$, $|a_E| < 2\sqrt{q}$, and
	\[\frac{2}{b^2} \left(1 - \sqrt{(qb^2 - 1)(b^2 - 1)}\right) < a_E < \frac{2}{a^2} \left(1 - \sqrt{(qa^2 - 1)(a^2 - 1)}\right).
\]
Note that for all sufficiently large $q$ there exists such an integer $a_E$ for which this inequality holds, and hence by the proof of \cite[Theorem 4.1]{Waterhouse} there exists an elliptic curve $E$ over $\F_q$ satisfying LI with trace of the Frobenius equal to $a_E$. It follows that
	\[a < 2\sqrt{\frac{q + 1 - a_E}{4q - a_E^2}} < b
\]
and as it is shown in \cite[\S 3]{Humphries} that
	\[B(E/\F_q) = 2\sqrt{\frac{q + 1 - a_E}{4q - a_E^2}},
\]
the result follows.
\end{proof}

\section{The Minimum of $\varphi(U)$}
\label{unitarysection}

Let $\U(N)$ denote the space of $N \times N$ unitary matrices, so that a matrix $U \in \U(N)$ has eigenvalues $e^{i \theta_1}, \ldots, e^{i \theta_N}$ with $-\pi \leq \theta_j \leq \pi$ for all $1 \leq j \leq N$. For real $\theta$, the characteristic polynomial $Z_U(\theta)$ of $U$ is defined to be
	\[\ZZ_U(\theta) = \det\left(I - U e^{-i \theta}\right) = \prod^{N}_{j = 1}{\left(1 - e^{i\left(\theta_j - \theta\right)}\right)}.
\]
Let
\begin{equation}\label{varphiunitary}
\varphi(U) = \sum^{N}_{j = 1}{\frac{1}{\left|{\ZZ_U}'\left(\theta_j\right)\right|}},
\end{equation}
so that
\begin{equation}\label{varphiunitarydiff}
\begin{split}
\varphi(U) = \varphi\left(\theta_1, \ldots, \theta_N\right) & = \sum^{N}_{j = 1}{\prod^{N}_{\substack{k = 1 \\ k \neq j}}{\frac{1}{\left|1 - e^{i \left(\theta_k - \theta_j\right)}\right|}}}	\\
&  = \frac{1}{2^{N - 1}} \sum^{N}_{j = 1}{\prod^{N}_{\substack{k = 1 \\ k \neq j}}{\left|\cosec\left(\frac{\theta_k - \theta_j}{2}\right)\right|}}.
\end{split}
\end{equation}
We prove the following generalisation of Proposition \ref{unitarysympthm}.

\begin{proposition}\label{unitarythm}
Let $U \in \U(N)$ be a unitary matrix with eigenvalues $e^{i \theta_1}, \ldots, e^{i \theta_N}$, and let $\varphi(U) = \varphi\left(\theta_1, \ldots, \theta_N\right)$ be as in \eqref{varphiunitary}. Then the global minimum of $\varphi$ occurs precisely at the set of points
	\[\left(\widetilde{\theta}_{\sigma(1)} + \phi, \ldots, \widetilde{\theta}_{\sigma(N)} + \phi\right),
\]
where $\sigma \in S_N$, $\phi$ is a one-dimensional translation modulo $2\pi$, and
	\[\left(\widetilde{\theta}_1, \ldots, \widetilde{\theta}_N\right) = \left(- \frac{(N - 1)\pi}{N}, - \frac{(N - 3) \pi}{N}, \ldots, \frac{(N - 1)\pi}{N}\right).
\]
Furthermore,
\begin{equation}\label{unitaryminimum}
\varphi\left(\widetilde{\theta}_{\sigma(1)} + \phi, \ldots, \widetilde{\theta}_{\sigma(N)} + \phi\right) = 1.
\end{equation}
\end{proposition}

From this, we obtain the result for the subgroup of unitary symplectic matrices, namely Proposition \ref{unitarysympthm}, by setting $N = 2g$ and restricting our values of $\left(\theta_1, \ldots, \theta_{2g}\right)$ to be such that $\theta_{j + g} = - \theta_j$ with $0 \leq \theta_j \leq \pi$ for each $1 \leq j \leq g$; note that this restriction means that we lose the one-dimensional translation invariance modulo $2\pi$ of the variables of $\varphi$.

One can interpret Proposition \ref{unitarythm} via a geometric argument. If $z_1, \ldots, z_N$ are $N$ points on the unit circle in the complex plane, then we may consider the product of the chord lengths of chords from a single point $z_j$ to the other $N - 1$ points. We can then think of $\varphi$ as the sum of the inverses of these products indexed by the starting points $z_j$. Intuitively, we would expect the product of chord lengths to be largest when averaged over the starting points when the $N$-tuple of points on the unit circle are evenly spaced; consequently, we would expect $\varphi$ to be smallest at this same $N$-tuple.

\begin{proof}[Proof of Proposition \ref{unitarythm}]
We first prove that \eqref{unitaryminimum} holds. It suffices to prove this when $\sigma$ is the identity and $\phi = 0$, as $\varphi$ is invariant under permutations and one-dimensional translations of the variables. From \eqref{varphiunitarydiff}, we have that
	\[\varphi\left(\widetilde{\theta}_1, \ldots, \widetilde{\theta}_N\right) = \sum^{N}_{j = 1}{\prod^{N}_{\substack{k = 1 \\ k \neq j}}{\frac{1}{\left|1 - e^{2 \pi i (k - j) / N}\right|}}},
\]
as $\theta_j = (2j - 1 - N) \pi / N$. We obtain \eqref{unitaryminimum} by noting that for any $1 \leq j \leq N$,
	\[\prod^{N}_{\substack{k = 1 \\ k \neq j}}{\frac{1}{\left|1 - e^{2 \pi i (k - j) / N}\right|}} = \prod^{N - 1}_{k = 1}{\frac{1}{\left|1 - e^{2 \pi i k / N}\right|}} = \frac{1}{N},
\]
where the last step follows by taking $x = 1$ in the identity
	\[\sum^{N - 1}_{j = 0}{x^j} = \prod^{N - 1}_{k = 1}{\left(x - e^{2 \pi i k / N}\right)}.
\]
To prove that $\varphi\left(\theta_1, \ldots, \theta_N\right) \geq 1$ for all $\left(\theta_1, \ldots, \theta_N\right) \in [-\pi,\pi]^N$, we first note that we may assume without loss of generality that $- \pi < \theta_1 < \ldots < \theta_N < \pi$ and that $\theta_1 = - (N - 1) \pi / N$, as $\varphi$ is invariant under permutations and one-dimensional translations modulo $2 \pi$ of the variables. By applying the arithmetic mean--geometric mean inequality to \eqref{varphiunitarydiff}, we have that
	\[\varphi\left(\theta_1, \ldots, \theta_N\right) \geq \frac{N}{2^{N - 1}} \prod^{N}_{\substack{j, k = 1 \\ j \neq k}}{\left|\cosec\left(\frac{\theta_k - \theta_j}{2}\right)\right|^{1/N}} = \frac{N}{2^{N - 1}} \prod^{N - 1}_{j = 1}{\prod^{N}_{k = 1}{\left(\cosec \omega_{jk}\right)^{1/N}}},
\]
where $\omega_{jk} = \omega_{jk}\left(\theta_1, \ldots, \theta_N\right) \in (0, \pi)$ is given by
	\[\omega_{jk} = \begin{cases}
\displaystyle \frac{\theta_{j + k} - \theta_k}{2} & \text{if $j + k \leq N$,} \vspace{0.2cm}	\\
\displaystyle \pi - \frac{\theta_k - \theta_{j + k - N}}{2} & \text{if $j + k > N$,}
\end{cases}\]
and we have used the fact that $|\cosec \theta| = \cosec |\theta|$ and that $\cosec \left(\pi - \theta\right) = \cosec \theta$ for $- \pi < \theta < \pi$, $\theta \neq 0$. As
	\[\prod^{N}_{k = 1}{\left(\cosec \omega_{jk}\right)^{1/N} = \exp\left(\frac{1}{N} \sum^{N}_{k = 1}{\log \cosec \omega_{jk}}\right)},
\]
and as the function $f(\theta) = \log \cosec \theta$ is convex on the interval $(0,\pi)$, Jensen's inequality implies that
	\[\varphi\left(\theta_1, \ldots, \theta_N\right) \geq \frac{N}{2^{N - 1}} \prod^{N - 1}_{j = 1}{\cosec \left(\frac{1}{N} \sum^{N}_{k = 1}{\omega_{jk}}\right)}.
\]
Now
\begin{equation}\label{telescopingsum}
\sum^{N}_{k = 1}{\omega_{jk}} = j \pi,
\end{equation}
as this is a telescoping sum, and consequently
	\[\varphi\left(\theta_1, \ldots, \theta_N\right) \geq \frac{N}{2^{N - 1}} \prod^{N - 1}_{j = 1}{\cosec \left(\frac{j \pi}{N}\right)} = N \prod^{N - 1}_{j = 1}{\frac{1}{\left|1 - e^{2 \pi i j / N}\right|}} = 1.
\]
Finally, the function $f(\theta) = \log \cosec \theta$ is strictly convex on $(0,\pi)$, so equality from the use of Jensen's inequality can only occur if for each fixed $1 \leq j \leq N$,
	\[\omega_{jk} = \omega_{jk'}
\]
for all $1 \leq k,k' \leq N$, which, together with \eqref{telescopingsum}, implies that
	\[\omega_{jk} = \frac{j \pi}{N}
\]
for all $1 \leq j,k \leq N$. As we assumed that $\theta_1 = - (N - 1)\pi/N$, it follows that equality can only hold when
	\[\left(\theta_1, \ldots, \theta_N\right) = \left(\widetilde{\theta}_1, \ldots, \widetilde{\theta}_N\right) = \left(- \frac{(N - 1) \pi}{N}, - \frac{(N - 3) \pi}{N}, \ldots, \frac{(N - 1) \pi}{N}\right).
\qedhere\]
\end{proof}

It is worth noting that this method also works for the more general function
	\[\sum^{N}_{j = 1}{\left|{\ZZ_U}'\left(\theta_j\right)\right|^{2k}}
\]
where $k < 0$, which was studied by Hughes, Keating, and O'Connell \cite{Hughes} for its relation to discrete moments of the derivative of the Riemann zeta function. They calculated the asymptotics for large $N$ of the integral
	\[\int_{\U(N)}{\sum^{N}_{j = 1}{\left|{\ZZ_U}'\left(\theta_j\right)\right|^{2k}} \, d\mu_{\Haar}(U)},
\]
where $\mu_{\Haar}$ is the Haar measure on $\U(N)$, and used this to conjecture the growth in the variable $T$ of the sum
	\[J_k(T) = \sum_{0 < \gamma \leq T}{\left|\zeta'\left(1/2 + i \gamma\right)\right|^{2k}},
\]
where we are assuming the Riemann hypothesis and the simplicity of the zeroes of $\zeta(s)$. The method of proof of Proposition \ref{unitarythm} shows that for $k < 0$, the global minimum of
	\[\sum^{N}_{j = 1}{\left|{\ZZ_U}'\left(\theta_j\right)\right|^{2k}}
\]
is $2^{2k + 1}$, and occurs at the points $\left(\widetilde{\theta}_{\sigma(1)} + \phi, \ldots, \widetilde{\theta}_{\sigma(N)} + \phi\right)$.

\subsection*{Acknowledgements}
The author would like to thank Jim Borger for his helpful advice and support, Byungchul Cha for his useful discussions on his work, and the anonymous referee for their many suggestions and corrections. Most of all, the author is indebted to Ruixiang Zhang for his sketch of a simple proof of Proposition \ref{unitarythm}.


\begin{thebibliography}{88}

\bibitem{Bateman} P.\ T.\ Bateman, J.\ W.\ Brown, R.\ S.\ Hall, K.\ E.\ Kloss, and Rosemarie M.\ Stemmler, ``Linear Relations Connecting the Imaginary Parts of the Zeros of the Zeta Function'', in \textit{Computers in Number Theory}, editors A.\ O.\ L.\ Atkin and B.\ J.\ Birch, Academic Press, London, 1971, 11--19.

\bibitem{Best} D.\ G.\ Best and T.\ S.\ Trudgian, ``Linear Relations of Zeroes of the Zeta-Function'', preprint, \href{http://arxiv.org/abs/1209.3843}{arXiv:math.NT/1209.3843} (18 September 2012), 12 pages.

\bibitem{Cha} Byungchul Cha, ``The Summatory Function of the M\"{o}bius Function in Function Fields'', submitted for publication, \href{http://arxiv.org/abs/1008.4711}{arXiv:math.NT/1008.4711v2} (14 November 2011), 16 pages.

\bibitem{Chavdarov} Nick Chavdarov, ``The Generic Irreducibility of the Numerator of the Zeta Function in a Family of Curves with Large Monodromy'', \textit{Duke Mathematical Journal} \textbf{87} (1997), 151--180.

\bibitem{Howe} Everett W.\ Howe, Enric Nart, and Christophe Ritzenthaler, ``Jacobians in Isogeny Classes of Abelian Surfaces over Finite Fields'', \textit{Annales de l'Institut Fourier} \textbf{59} (2009), 239--289.

\bibitem{Hughes} C. P. Hughes, J. P. Keating, and Neil O'Connell, ``Random Matrix Theory and the Derivative of the Riemann Zeta-Function'', \textit{Proceedings of the Royal Society of London Serial A} \textbf{456} (2000), 2611--2627.

\bibitem{Humphries} Peter Humphries, ``On the Mertens Conjecture for Elliptic Curves over Finite Fields'', to appear in \textit{Bulletin of the Australian Mathematical Society}, \href{http://dx.doi.org/10.1017/S0004972712001116}{dx.doi.org/10.1017/S0004972712001116} (28 February 2013), 15 pages.

\bibitem{Ingham} A.\ E.\ Ingham, ``On Two Conjectures in the Theory of Numbers'', \textit{American Journal of Mathematics} \textbf{64} (1942), 313--319.

\bibitem{Katz} Nicholas M.\ Katz and Peter Sarnak, \textit{Random Matrices, Frobenius Eigenvalues, and Monodromy}, American Mathematical Society Colloquium Publications \textbf{45}, American Mathematical Society, Providence, 1999.

\bibitem{Kotnik} Tadej Kotnik and Herman te Riele, ``The Mertens Conjecture Revisited'', in \textit{Algorithmic Number Theory; 7th International Symposium, ANTS-VII; Berlin, Germany, July 2006; Proceedings}, editors Florian Hess, Sebastian Pauli, and Michael Pohst, Lecture Notes in Computer Science \textbf{4076}, Springer, Berlin, 2006, 156--167.

\bibitem{Kowalski} Emmanuel Kowalski, ``The Large Sieve, Monodromy, and Zeta Functions of Algebraic
Curves, 2: Independence of the Zeros'', \textit{International Mathematics Research Notices} (2008), Article ID rnn091, 57 pages.

\bibitem{Mertens} F.\ Mertens, ``\"{U}ber eine zahlentheoretische Funktion'', \textit{Sitzungs\-berichte der Kaiserlichen Akademie der Wissenschaften, Mathematisch-Natur\-wissen\-schaftliche Klasse, Abteilung 2a} \textbf{106} (1897), 761--830.

\bibitem{Ng} Nathan Ng, ``The Distribution of the Summatory Function of the M\"{o}bius Function'', \textit{Proceedings of the London Mathematical Society} \textbf{89} (2004), 361--389.

\bibitem{Odlyzko} A.\ M.\ Odlyzko and H.\ J.\ J.\ te Riele, ``Disproof of the Mertens Conjecture'', \textit{Journal f\"{u}r die Reine und Angewandte Mathematik} \textbf{357} (1985), 138--160.

\bibitem {Rosen} Michael Rosen, \textit{Number Theory in Function Fields}, Graduate Texts in Mathematics \textbf{210}, Springer, New York, 2002.

\bibitem{vonSterneck} R.\ D.\ von Sterneck, ``Neue empirische Daten \"{u}ber die zahlentheoretischen Funktion $\sigma(n)$'', in \textit{Proceedings of the Fifth International Congress of Mathematics, Cambridge, 22--28 August 1912, Volume 1}, editors E.\ W.\ Hobson and A.\ E.\ H.\ Love, Cambridge, 1913, 341--343.

\bibitem{Stieltjes} T.\ J.\ Stieltjes, Lettre \`{a} Hermite de 11 juillet 1885, Lettre \#79, in \textit{Correspondance d'Hermite et de Stieltjes, Tome 1}, editors B.\ Baillaud and H.\ Bourget, Paris, Gauthier--Villars, 1905, 160--164.

\bibitem{Waterhouse} William C.\ Waterhouse, ``Abelian Varieties over Finite Fields'', \textit{Annales Scientifiques de l'\'{E}cole Normale Sup\'{e}rieure, S\'{e}rie 4} \textbf{2} (1969), 521--560.

\bibitem{Weil} Andr\'{e} Weil, ``Sur les Fonctions Alg\'{e}briques \`{a} Corps de Constantes Fini'', \textit{Comptes Rendus de l'Acad\'{e}mie des Sciences de Paris, Serie I.\ Mathematique} \textbf{210} (1940), 592--594.

\end{thebibliography}
\end{document}